\documentclass[twocolumn]{autart}    
\usepackage{graphicx,amsmath,amsfonts,amssymb,mathrsfs,bm,tabularx,array,enumerate,pifont,color,harvard}

\newtheorem{lemma}{Lemma}
\newtheorem{theorem}{Theorem}

\newtheorem{assumption}{Assumption}

\newtheorem{definition}{Definition}

\DeclareMathOperator{\sign}{sgn}

\DeclareMathOperator{\argmin}{argmin}
\DeclareMathOperator{\VI}{VI}
\DeclareMathOperator{\SOL}{SOL}
\DeclareMathOperator{\rint}{rint}

\newenvironment{proof}{Proof.}{\hfill $\square$}
\begin{document}

\begin{frontmatter}

\title{Distributed Nash equilibrium seeking for aggregative games with coupled constraints \thanksref{footnoteinfo}} 

\thanks[footnoteinfo]{This paper was not presented at any IFAC
meeting. Corresponding author Y.~Hong. Tel. +86-10-82541824.
Fax +86-10-82541832.}

\author[CAS]{Shu Liang}\ead{sliang@amss.ac.cn},    
\author[TR]{Peng Yi}\ead{peng.yi@utoronto.ca},               
\author[CAS]{Yiguang
Hong}\ead{yghong@iss.ac.cn}  

\address[CAS]{Key Laboratory of Systems and Control, Academy of
Mathematics and Systems Science, Chinese Academy of Sciences,
Beijing, 100190, China}                   
\address[TR]{Department of Electrical \& Computer Engineering, University of Toronto, Canada}             


\begin{abstract}                          
In this paper, we study a distributed continuous-time design for
aggregative games with coupled constraints in order to seek the generalized
Nash equilibrium by a group of agents via simple local
information exchange.   To solve the problem, we propose a distributed
algorithm based on projected dynamics and non-smooth tracking dynamics, even for the case when the interaction topology of the multi-agent network
is time-varying.  Moreover, we prove the
convergence of the non-smooth algorithm for the distributed game by taking advantage of its special structure and also combining the techniques of the variational inequality and Lyapunov function.
\end{abstract}

\end{frontmatter}

\section{Introduction}
The seek of generalized Nash equilibrium for non-cooperative games with coupled constraints has been widely investigated due to various applications in natural/social science and engineering (such as
telecommunication power allocation and cloud computation \cite{Pang2008Distributed,Ardagna2013Generalized}).  Significant theoretic and algorithmic achievement has been done, referring to \cite{Pavel2007Extension,Altman2009Constrained,Arslan2015On} and \cite{Facchinei2010Generalized}.

Distributed equilibrium seeking algorithms guide a group of players or agents to cooperatively achieve
the Nash equilibrium (NE), based on players' local information and information exchange
between their neighbors in a network.  The NE seeking may be viewed
as an extension of distributed optimization problems, which have been widely studied recently (see
\cite{Nedic2009Distributed,Shi2013reaching,Kia2015Distributed}), and on the other hand, distributed optimization problems can be handled with a game-theoretic approach \cite{Li2013Designing}.
In fact, in the study of complicated behaviors of strategic-interacted players in large-scale networks, it is quite natural to investigate game theory in a distributed way. For
example, distributed convergence to NE of zero-sum games over two
subnetworks was obtained in \cite{Lou2016Nash}. Moreover, a distributed fictitious play algorithm was proposed in \cite{Swenson2015Empirical}, while a gossip-based approach was employed for seeking an NE of noncooperative games in \cite{Salehisadaghiani2016Distributed}.

Aggregative games have become an important type of games since
the well-known Cournot model was proposed, and have recently been
studied in the literature, referring to
\cite{Jensen2010Aggregative,Cornes2012Fully}, for its broad application in public environmental models \cite{Cornes2016Aggregative}, congestion control of
communication networks \cite{Barrera2015Dynamic}, and demand
response management of power systems \cite{Ye2017Game}.  Usually,
linear aggregation functions and quadratic cost functions in such games were
considered, for example, in \cite{Parise2015Network,Paccagnan2016Distributed,Ye2017Game}. Also, 
a recent result was given for distributed discrete-time algorithms to seek the NE of an aggregative game with
time-varying topologies in \cite{Koshal2016Distributed}.

The objective of this paper is to develop a novel distributed
continuous-time algorithm for nonlinear aggregative games with
linear coupled constraints and time-varying topologies.     In recent
years, continuous-time algorithms for distributed optimization
become more and more popular \cite{Shi2013reaching,Kia2015Distributed,Yi2016Initialization},
partially because they may be easily implemented in continuous-time
or hybrid physical systems. However, ideas and approaches for
continuous-time design may not be the same as those for the discrete-time one.
Thanks to various well-developed continuous-time methods,
distributed continuous-time algorithms or protocols keep being
constructed, but the (convergence) conditions may be different from
those in discrete-time cases.

In our problem setup, every player tries to optimize its local cost function
by updating its local decision variable. The cost function depends on not only the local variable but also a nonlinear aggregation. Moreover, feasible decision variables of players are coupled by linear constraints. Existing distributed algorithms for aggregative games \cite{Koshal2016Distributed,Ye2017Game} cannot solve our problems since they did not consider coupled constraints. The contribution of this paper can be summarized as follows:
\begin{itemize}
\item The aggregative game model in this paper generalizes the previous ones in \cite{Paccagnan2016Distributed,Ye2017Game} by
allowing nonlinear aggregation term and non-quadratic cost
functions, and also those in \cite{Koshal2016Distributed} by considering coupled constraints. In addition, the considered game can be non-potential.
\item Inspired from distributed average tracking dynamics and
projected primal-dual dynamics, we take advantage of continuous-time techniques to solve the distributed problem. With the new idea, our algorithm is described as a {\em non-smooth} multi-agent system with two
interconnected dynamics: a projected gradient one for the equilibrium seeking, and a consensus one for the synchronization of the
aggregation and the dual variables. In addition, our algorithm need not solve the best response subproblems, different from those in \cite{Parise2015Network}, and can keep private some information about the cost functions, local decisions, and constraint coefficients.
\item We provide a method to prove the correctness and convergence of the continuous-time algorithm by combining the techniques from variational inequality theory and Lyapunov stability theory.
\end{itemize}

{\em Notations}: Denote $\mathbb{R}^n$ as the
$n$-dimensional real vector space; denote $\mathbf{1}_n =  (1,...,1)^{T} \in \mathbb{R}^n$, and $\mathbf{0}_n = (0,...,0)^{T} \in \mathbb{R}^n$. Denote $col(x_1,...,x_n) = (x_1^{T},...,x_n^{T})^{T}$ as the column vector
stacked with column vectors $x_1,...,x_n$, $\|\cdot\|$ as the Euclidean norm, and $I_n\in \mathbb{R}^{n\times n}$ as the identity matrix.   Denote $\nabla f$ as the gradient vector of a function $f$ and $\mathcal{J}F$ as the Jacobian matrix of a map $F$. 
Let $C_1\pm C_2 = \{z_1 \pm z_2\,|\,z_1\in C_1, z_2\in C_2\}$ be the Minkowski sum/minus of sets $C_1$ and $C_2$, and $\rint(C)$ be the relative interior of a convex set $C$ \cite[page 25 and page 64]{Rockafellar1998Variational}.

\section{Preliminaries}
In this section, we give some preliminary knowledge related to convex analysis, variational inequality, and graph theory.

A set $C \subseteq \mathbb{R}^n$ is {\em convex} if $\lambda z_1
+(1-\lambda)z_2\in C$ for any $z_1, z_2 \in C$ and $0\leq \lambda \leq 1$.
For a closed convex set $C$, the {\em projection} map $P_{C}:\mathbb{R}^n \to C$ is defined as
\begin{equation*}
P_{C}(x) \triangleq \argmin_{y\in C} \|x-y\|.
\end{equation*}
The following two basic properties hold:
\begin{align}
\label{eq:pro_projection}
&(x-P_{C}(x))^{T}(P_{C}(x)-y) \geq 0, \quad \forall \, y\in C\text{,}\\
\label{eq:pro_Lip}
&\|P_{C}(x)-P_{C}(y)\|\leq \|x-y\|, \quad \forall \, x,y\in \mathbb{R}^n\text{.}
\end{align}

For $x\in C$, the {\em tangent cone} to $C$ at $x$ is
\begin{multline*}
\mathcal{T}_C(x) \triangleq \{\lim_{k\to\infty}\frac{x_k-x}{t_k}\,|\,x_k\in C, t_k>0, \\
\text{ and }x_k \to x, t_k\to 0\}\text{.}
\end{multline*}
and the {\em normal cone} to $C$ at $x$ is
\begin{equation*}
\mathcal{N}_C(x) \triangleq \{v\in \mathbb{R}^n\,|\,v^T(y -x) \leq 0, \text{ for all } y\in C\}.
\end{equation*}

\begin{lemma}\label{lemma:tangent_cone}
{\em \cite[Theorem 6.42]{Rockafellar1998Variational}}
Let $C_1$ and $C_2$ be two closed convex subsets of $\mathbb{R}^n$. If $0\in \rint(C_1-C_2)$, then
\begin{equation*}
\mathcal{T}_{C_1\cap C_2}(x) = \mathcal{T}_{C_1}(x)\cap \mathcal{T}_{C_2}(x), \,\forall \, x\in C_1\cap C_2.
\end{equation*}
\end{lemma}


A function $f:\mathbb{R}^{n}\to \mathbb{R}$ is {\em convex} if
$f(\lambda z_1+(1-\lambda)z_2)\leq \lambda f(z_1) +
(1-\lambda)f(z_2)$ for any $z_1, z_2 \in C$ and $0\leq \lambda \leq
1$. A map $F:\mathbb{R}^n\rightarrow \mathbb{R}^n$ is said to be
{\em monotone (strictly monotone)}  on a set $\Omega$ if
$(x-y)^{T}(F(x)-F(y))\geq 0\,(> 0)$ for all $x,y\in\Omega$ and
$x\neq y$. A differentiable map $F$ is monotone if and only if the Jacobian matrix $\mathcal{J} F(x)$ (not necessarily symmetric) is positive semidefinite for each $x$ \cite[Theorem 12.3]{Rockafellar1998Variational}.

Given a subset $\Omega \subseteq \mathbb{R}^n$ and a map $F: \Omega \to \mathbb{R}^n$, the {\em variational inequality}, denoted by $\VI(\Omega,F)$,
is to find a vector $x\in \Omega$ such that
\begin{equation*}
(y-x)^TF(x) \geq 0,\quad \forall\,y\in \Omega,
\end{equation*}
and the set of solutions to this problem is denoted by
$\SOL(\Omega,F)$ \cite{Facchinei2003Finite}. When $\Omega$ is closed and convex, the
solution of $\VI(\Omega,F)$ can be equivalently reformulated via
projection as follows:
\begin{equation}\label{eq:SOL2Pro}
x\in \SOL(\Omega,F) \Leftrightarrow x = P_\Omega(x-F(x))\text{.}
\end{equation}

\begin{lemma}\label{lemma:existence}
{\em \cite[Corollary 2.2.5, and Theorem
2.2.3]{Facchinei2003Finite}}
Consider $\VI(\Omega,F)$, where the set $\Omega \subset
\mathbb{R}^n$ is convex and the map $F:\Omega\to \mathbb{R}^n$ is
continuous.  The following two statements hold:
\begin{enumerate}[1)]
\itemsep = 0pt \parskip = 0pt
\item if $\Omega$ is compact, then $\SOL(\Omega,F)$ is nonempty and compact;
\item if $\Omega$ is closed and $F(x)$ is strictly monotone, then $\VI(\Omega,F)$ has at most one solution.
\end{enumerate}
\end{lemma}

The following lemma about a regularized gap function is important for our results.

\begin{lemma}\label{lemma:projection_differential}
{\em \cite{Fukushima1992Equivalent}}
Let $F: \mathbb{R}^n\to \mathbb{R}^n$ be a differentiable map and $H(x) = P_\Omega(x-F(x))$. Define $g:\mathbb{R}^n\to \mathbb{R}$ as
\begin{equation*}
g(x) = (x-H(x))^{T}F(x)-\frac{1}{2}\|x-H(x)\|^2\text{.}
\end{equation*}
Then $g(x)\geq 0$ is differentiable and its gradient is
\begin{equation*}
\nabla g(x) = F(x) + (\mathcal{J} F(x) - I_n)(x-H(x))\text{.}
\end{equation*}
\end{lemma}

Furthermore, it is known that the information exchange among agents
can be described by a graph. A graph with node set $\mathcal{V}$ and edge set $\mathcal{E}$ is written as $\mathcal{G}=(\mathcal{V},\,\mathcal{E})$ \cite{Godsil01}. If agent $i\in \mathcal{V}$ can receive information
from agent $j\in \mathcal{V}$, then $(j,\,i) \in \mathcal{E}$ and agent $j$ belongs
to agent $i$'s neighbor set $\mathcal{N}_i= \{ j \,|\,
(j,\,i) \in \mathcal{E}\}$. $\mathcal{G}$ is said to be {\em undirected}
if $(i,\,j)\in \mathcal{E} \Leftrightarrow (j,\,i)\in \mathcal{E}$, and
$\mathcal{G}$ is said to be {\em connected} if any two nodes in $\mathcal{V}$ are connected by a path (a sequence of distinct nodes in which any consecutive pair of nodes share an edge).



\section{Problem Formulation}\label{nea}
Consider an $N$-player aggregative game with coupled constraints as follows. For $i\in \mathcal{V} \triangleq \{1,...,N\}$, the $i$th player aims to minimize its cost function $J_i(x_i,x_{-i}): \Omega \to \mathbb{R}$ by choosing the local decision variable $x_i$ from a local strategy set $\Omega_i \subset \mathbb{R}^{n_i}$, where $x_{-i} \triangleq
col(x_1,...,x_{i-1},x_{i+1},...,x_{N})$, $\Omega \triangleq
\Omega_1 \times \cdots \times \Omega_N \subset \mathbb{R}^n$ and $n = \sum_{i\in\mathcal{V}} n_i$. The {\em strategy profile} of this game is $x \triangleq col(x_1,...,x_N)\in
\Omega$. The {\em aggregation} map $\sigma: \mathbb{R}^{n}\to \mathbb{R}^m$, to specify the cost function as $J_i(x_i,x_{-i}) = \vartheta_i(x_i,\sigma(x))$ with a function $\vartheta_i: \mathbb{R}^{n_i+m} \to \mathbb{R}$, is defined as
\begin{equation}\label{eq:sigma}
\sigma(x) \triangleq \frac{1}{N}\sum_{i=1}^N \varphi_i(x_i)\text{,}
\end{equation}
where $\varphi_i : \mathbb{R}^{n_i} \to \mathbb{R}^m$ is a (nonlinear) map for the local contribution to the aggregation. In addition, the {\em feasible strategy set} of this game is $\mathcal{K} = \Omega \cap \mathcal{X}$, where $\mathcal{X}$ is a set for {\em linear coupled constraints}, defined as
\begin{multline}\label{eq:resource_allocation}
\mathcal{X} \triangleq \{x\in \mathbb{R}^{n}\,|\,\sum_{i\in\mathcal{V}}A_ix_i = \sum_{i\in\mathcal{V}}b_i\} \\
= \{x\in \mathbb{R}^{n}\,|\,Ax-b = \mathbf{0}_l\} \text{,}
\end{multline}
for some $A_i \in \mathbb{R}^{l\times n_i}, \, b_i\in \mathbb{R}^l$, $A = [A_1,...,A_N]$, and $b = \sum_{i\in\mathcal{V}}b_i$.

For such {\em games with coupled constraints}, the following concept of generalized Nash equilibrium is considered.

\begin{definition}\label{defn:GNE}
A strategy profile $x^*$ is said to be a {\em
generalized Nash equilibrium} (GNE) of the game if
\begin{equation}\label{eq:K}
J_i(x_i^*,x_{-i}^*) \leq J_i(y,x_{-i}^*), \, \forall \, y : (y,x_{-i}^*)\in \mathcal{K}, i\in \mathcal{V}.
\end{equation}
\end{definition}
Condition \eqref{eq:K} means that all players simultaneously take their own
best (feasible) responses at $x^*$, where no player can further
decrease its cost function by changing its decision variable unilaterally.

Moreover, a strategy profile is said to be a {\em variational
equilibrium}, or variational GNE, if it is a solution of
$\VI(\mathcal{K},F)$, where the map $F(x):\mathbb{R}^n\to \mathbb{R}^n$ is
defined as
\begin{equation}\label{eq:F}
F(x) \triangleq col\{\nabla_{x_1} J_{1}(\cdot,x_{-1}),..., \nabla_{x_N} J_{N}(\cdot,x_{-N})\}\text{.}
\end{equation}
The variational GNEs are well-defined due to the following result.

\begin{lemma}\label{lemma:variational_GNE}
{\em \cite[Theorem 3.9]{Facchinei2010Generalized}}
If $\mathcal{K}$ is convex, every solution of the $\VI(\mathcal{K},F)$ is also a GNE.
\end{lemma}


The following assumption and theorem are associated with the game and the variational GNEs.
\begin{assumption}\label{assum:basic}
~
\begin{itemize}
\itemsep = 0pt \parskip = 0pt
\itemsep = 0pt \parskip = 0pt
\item {\em Smoothness:} $\forall\,i\in \mathcal{V}$, the cost function $J_{i}(x_i,x_{-i})$ is twice continuously differentiable.
\item {\em Monotonicity:} $F(x)$ in \eqref{eq:F} is strictly monotone.
\item {\em Feasibility:} $\Omega$ is compact and convex.
\item {\em Constraint qualification:} $0\in \rint(\Omega-\mathcal{X})$.
\end{itemize}
\end{assumption}

\begin{theorem}\label{thm:G2VI}
Under Assumption \ref{assum:basic}, the considered game admits a unique variational GNE.
\end{theorem}

\begin{proof}
According to Lemma \ref{lemma:existence}, the smoothness and feasibility in Assumption
\ref{assum:basic} guarantee the existence of a variational GNE, and the monotonicity in Assumption
\ref{assum:basic} guarantees the uniqueness.
\end{proof}

The constraint qualification in Assumption \ref{assum:basic} is quite mild and can be easily verified. In fact, it suffices
to check whether the set $\mathcal{K}$ has nonempty relative interior, {\em
i.e.}, $\rint(\mathcal{K}) \neq \emptyset$, because in this case,
\begin{equation*}
0\in \rint(\mathcal{K}-\mathcal{K})\subseteq \rint(\Omega-\mathcal{X})\text{.}
\end{equation*}

In the distributed design for our aggregative game, the communication topology for each player to exchange information is assumed as follows.

\begin{assumption}\label{assum:graph}
The (time-varying) graph $\mathcal{G}(t)$ is undirected and connected.
\end{assumption}

Then we formulate our problem.

{\bf Problem 1:} Design a distributed algorithm to seek the
variational GNE for the considered aggregative game with coupled constraints.

Here are some remarks about our formulation.
\begin{itemize}
\item Our formulation is quite similar to that in \cite{Koshal2016Distributed}, but we further consider linear coupled constraints and study its distributed continuous-time GNE seeking algorithm.
\item Players may not measure the values of the aggregation directly, in contrast to \cite{Parise2015Network}.
\item Our aggregation function can
be nonlinear and the cost functions can be non-quadratic, which are different from \cite{Paccagnan2016Distributed} and \cite{Ye2017Game}. In addition, $\mathcal{J}F(x)$ can be asymmetric\begin{footnote}{A game is a potential game if there is a function $P(x)$ such that $\nabla P(x) = F(x)$ \cite{Ye2017Game}. This equation holds if and only if the Jacobian matrix $\mathcal{J}F(x)$ is symmetric \cite[Theorem 1.3.1]{Facchinei2003Finite}.}\end{footnote}, which was also discussed in \cite{Paccagnan2016Distributed}.
\end{itemize}

\section{Main Results}

In this section, we first propose our distributed algorithm and then analyze its correctness and convergence.

\subsection{Distributed Algorithm}
Let $\alpha, \beta, \gamma>0$ be some constants satisfying $\alpha>(N-1)\bar{f}_1$ and $\beta > \gamma(N-1)\bar{f}_2$, where
\begin{align}
\label{eq:f1}
\bar{f}_1 & = \sup_{i\in \mathcal{V}} \big(\sup_{x_i\in \Omega_i}\|\nabla \varphi_i(x_i)\| \sup_{y,z\in\Omega}\|y-z\|\big)\text{,}\\
\label{eq:f2}
\bar{f}_2 & = \sup_{i\in \mathcal{V}} \big(\sup_{x_i\in \Omega_i}\|A_ix_i-b_i\|\big)\text{.}
\end{align}
For $i\in\mathcal{V}$, define the map $G_i(\cdot): \mathbb{R}^{n_i+m}\to \mathbb{R}^{n_i}$ as
\begin{align}\label{eq:Gi}
G_i(x_i,y_i) & \triangleq \nabla_{x_i} J_{i}(\cdot,x_{-i})|_{\sigma(x) = y_i}\\
\nonumber
& = (\nabla_{x_i} \vartheta_i(\cdot,\sigma) + \frac{1}{N} \nabla_{\sigma} \vartheta_i(x_i,\cdot)^T \nabla \varphi_i)|_{\sigma = y_i}\text{.}
\end{align}

Then the distributed continuous-time algorithm to solve Problem 1 is designed as follows:
\begin{equation}\label{eq:algorithm2}
\left\{\begin{aligned}
\dot{x}_i & = P_{\Omega_i}(x_i-G_i(x_i,\eta_i) - \frac{\gamma}{N}A_i^T\lambda_i)-x_i\\
\dot{\lambda}_i & = \beta \sum_{j\in\mathcal{N}_i}\sign(\lambda_j-\lambda_i) + \gamma(A_ix_i-b_i)\\
\dot{\zeta}_i & = \alpha \sum_{j\in \mathcal{N}_i}\sign(\eta_j - \eta_i)\\
\eta_i & =  \zeta_i + \varphi_i(x_i)
\end{aligned}\right.
\end{equation}
where $\sign(\cdot)$ is the sign function, $\lambda_i(t)\in \mathbb{R}^l, \zeta_i(t) \in\mathbb{R}^m$. The initial conditions of algorithm \eqref{eq:algorithm2} are provided as follows:
\begin{align}
x_i(0) \in \Omega_i, \quad \lambda_i(0) = A_ix_i(0) - b_i, \quad \zeta_i(0) = \mathbf{0}_m\text{.}
\end{align}

To set the parameters $\alpha, \beta, \gamma$ in algorithm \eqref{eq:algorithm2} needs the values of $\bar{f}_1, \bar{f}_2$ in \eqref{eq:f1} and \eqref{eq:f2}, which involves additional distributed calculation as follows: take variables $z_i$ with $z_i(0) = w_i$ for $i \in \mathcal{V}$, and update them by $z_i(k+1) = \sup\{z_i(k), z_j(k), j\in \mathcal{N}_i\}$; in this way, one can obtain $\sup\{w_i,i\in \mathcal{V}\}$ within $N-1$ steps.


\begin{rem}
The design idea for this continuous-time algorithm
\eqref{eq:algorithm2} is totally different from that given in
\cite{Koshal2016Distributed} for discrete-time algorithms. Note that $\eta_i$ in our algorithm is to estimate the value of aggregation $\sigma(x)$, while $\lambda_i$ is to estimate a dual variable associated with the coupled constraints.  Moreover, our algorithm is fully distributed and also preserves some privacy because the information, such as the local cost functions, decision variables, and coefficients of the coupled constraints, need not be shared.
\end{rem}

The solution of \eqref{eq:algorithm2} (with a discontinuous righthand
side) can be well defined in the Filippov sense, which is unique and
absolutely continuous.   For convenience, we will not mention ``in
the Filippov sense'' in the sequel when there is no confusion.

\subsection{Convergence Analysis}
Here we give some correctness and convergence analysis for our proposed algorithm.

First of all, we get the following result by extending \cite[Theorem 3]{Chen2012Distributed}.

\begin{lemma} \label{lemma:finite_time}
Under Assumption \ref{assum:graph}, if
\begin{equation*}
\alpha > (N-1) \bar{f}, \quad \bar{f} \geq \sup_{t\in [0,\infty)} \|\dot{r}_i(t)\|, \forall \, i \in \mathcal{V}
\end{equation*}
then the following system
\begin{equation}\label{eq:lem_algorithm1}
\left\{\begin{aligned}
\dot{\mu}_i(t) & =\alpha \sum_{j\in \mathcal{N}_i(t)}\sign[\nu_j(t) - \nu_i(t)]\\
\nu_i(t) & = \mu_i(t) + r_i(t),\quad \mu_i(0) = 0
\end{aligned}\right.
\end{equation}
can make $\lim_{t\to +\infty} \nu_i(t)- \frac{1}{N}\sum_{k=1}^N r_k(t)$ $=0$ for any $i\in\mathcal{V}$ with an exponential convergence rate.
\end{lemma}

\begin{proof}
Since $\mathcal{G}(t)$ is connected and $\mu_i(t)$ is absolutely continuous, $\sum_{i=1}^N\dot{\mu}_i(t) =
0$ for almost all $t\geq 0$, which implies $\sum_{i=1}^N\mu_i(t) = \sum_{i=1}^N\mu_i(0) = 0$. Therefore, $\sum_{i=1}^N \nu_i(t) = \sum_{i=1}^N r_i(t)$.   Moreover, it is not hard to obtain
\begin{itemize}
\item $\sum_{i=1}^N\nu_i\sum_{j\in\mathcal{N}_i}\sign(\nu_j - \nu_i) = \frac{1}{2} \sum_{(i,j)\in\mathcal{E}(t)}|\nu_i -
\nu_j|$;
\item for any $1\leq k,l\leq N$, $|\mu_k - \mu_l|\leq \frac{1}{2}\sum_{(i,j)\in\mathcal{E}(t)}|\nu_i - \nu_j|$;
\item $\sum_{i=1}^N|\nu_i - \frac{1}{N}\bm{1}^T\nu|\leq \frac{1}{N}\sum_{i=1}^N\sum_{j=1,j\neq i}^N|\nu_i-\nu_j| \leq \frac{N-1}{2}\sum_{(i,j)\in\mathcal{E}(t)}|\nu_i - \nu_j|$.
\end{itemize}
Define $\epsilon(t) \triangleq \max_{i,j\in\mathcal{V}}|\nu_i(t)-\nu_j(t)|$ and
\begin{equation*}
V(t) \triangleq \frac{1}{2} \|\nu(t) -
\frac{1}{N}\mathbf{1}\mathbf{1}^T\nu(t)\|^2.
\end{equation*}
Clearly, $\epsilon(t) \leq \frac{1}{2} \sum_{(i,j)\in\mathcal{E}(t)}|\nu_i(t) -
\nu_j(t)|$ and $V(t) = \frac{1}{2}\sum_{i=1}^N|\nu_i(t)- \frac{1}{N}\sum_{k=1}^N r_k(t)|^2 \geq 0$. Since $V(t)$ is absolutely continuous, we have that, for almost all $t>0$,
\begin{align*}
\dot{V}(t) & \leq - \frac{\alpha-(N-1)\bar{f}}{2}\sum_{(i,j)\in\mathcal{E}(t)}|\nu_i(t) -
\nu_j(t)|\\
& \leq -(\alpha-(N-1)\bar{f}) \epsilon(t)  \leq 0 \text{.}
\end{align*}
Because $\epsilon(t)\geq 0$ is absolutely continuous and $(\alpha-(N-1)\bar{f})\int_{0}^{+\infty}\epsilon(t) \leq V(0) < +\infty$, $\epsilon(t)\to 0$ as $t\to +\infty$. Then $V(t) \leq \epsilon(t)$ and $\dot{V}(t) \leq - (\alpha-(N-1)\bar{f}) V(t)$ for almost all $t\in [\tilde{t}, +\infty)$ with a sufficient large $\tilde{t}$, which implies the conclusion.
\end{proof}

Next, we give the following result, whose proof is quite straightforward by Lemma \ref{lemma:finite_time}.

\begin{lemma}\label{thm:Equivalent_algorithm}
Consider algorithm \eqref{eq:algorithm2} under Assumption \ref{assum:graph}.
Let $\sigma(x)$ be in \eqref{eq:sigma} and define
\begin{equation}\label{eq:barlambda}
\bar{\lambda}(t) = \frac{\gamma}{N}\left(Ax(0) -b +\int_0^t(Ax(\tau)-b)d\tau\right)\text{.}
\end{equation}
Then, for any $i\in \mathcal{V}$, $\|\eta_i(t) - \sigma(x(t))\| \to 0$ and $\|\lambda_i(t) - \bar{\lambda}(t)\|\to 0$ exponentially.
\end{lemma}


Let $X^*\times \Lambda^*$ be the solution set of the following equation with respect to $(x,\bar{\lambda})$
\begin{equation}\label{eq:KKT}
\left\{\begin{aligned}
0 & = P_{\Omega}(x-F(x) - \frac{\gamma}{N}A^T\bar{\lambda})-x\\
0 & = \frac{\gamma}{N}(Ax - b)
\end{aligned}\right.\text{,}
\end{equation}
and let $X^+\times \Lambda^+$ be the positive limit set of $(x(t),\bar{\lambda}(t))$, where $x(t)$ is in \eqref{eq:algorithm2} and $\bar{\lambda}(t)$ is in \eqref{eq:barlambda}. 

\begin{lemma}\label{eq:positive_limit}
If $\lim_{t\to+\infty}\dot{x}(t) = 0$ and $\lim_{t\to +\infty}\dot{\bar{\lambda}}(t) = 0$, then $X^+\times \Lambda^+ \subseteq X^*\times \Lambda^*$.
\end{lemma}

\begin{proof}
If $X^+\times \Lambda^+ \neq \emptyset$, then, for any
$(x^+,\bar{\lambda}^+) \in X^+\times \Lambda^+$, there exists
$\{t_q\}_{q=1}^{+\infty}$ such that $\lim_{q\to+\infty}x(t_q) = x^+$
and $\lim_{q\to+\infty}\bar{\lambda}(t_q) = \bar{\lambda}^+$. By
taking the limit of $t_q$ to \eqref{eq:algorithm2} and
\eqref{eq:barlambda}, one obtains according to Lemma
\ref{thm:Equivalent_algorithm} that $(x^+,\bar{\lambda}^+)$ is a
solution of equation \eqref{eq:KKT}. Thus, the conclusion follows.
\end{proof}

The next result reveals another property of $X^*\times \Lambda^*$, related to the variational GNE.

\begin{theorem}\label{thm:KKT}
Under Assumption \ref{assum:basic}, $x^*$ is the variational GNE if
and only if there exists some $\bar{\lambda}^*\in \mathbb{R}^l$ such that
$(x^*, \bar{\lambda}^*) \in X^*\times \Lambda^*$.
\end{theorem}

\begin{proof}
Sufficiency. Suppose $(x^*, \bar{\lambda}^*) \in X^*\times
\Lambda^*$. Then it follows from the definition of projection
operator that $x^* \in \Omega$. Moreover, $x^*\in \mathcal{X}$ since
$Ax^*-b = 0$. By \eqref{eq:SOL2Pro}, $x^* \in
\SOL(\Omega, F(x)+ \frac{\gamma}{N}A^T\bar{\lambda}^*)$, {\em i.e.},
\begin{equation}\label{eq:inequality1}
(x-x^*)^T(F(x^*)+\frac{\gamma}{N}A^T\bar{\lambda}^*) \geq 0,\,\forall\, x\in \Omega\text{.}
\end{equation}
Because $\mathcal{K}\subset \Omega$ and
$(x^*)^TA^T\bar{\lambda}^* = b^T\bar{\lambda}^*$, we have
\begin{multline*}
(x-x^*)^TF(x^*) \geq -\frac{\gamma}{N}(x-x^*)^TA^T\bar{\lambda}^* \\
= -\frac{\gamma}{N}(\bar{\lambda}^*)^T(Ax-b) = 0,\,\forall\, x\in \mathcal{K}\text{.}
\end{multline*}
Thus, $\{x^*\} = \SOL(K,F)$, {\em i.e.}, $x^*$ is the variational GNE.

Necessity. Suppose that $x^*$ is the variational GNE. We claim that there exists some $\bar{\lambda}^*$ such that
\begin{equation}\label{eq:inclusion_normal}
-F(x^*) \in \frac{\gamma}{N}A^T\bar{\lambda}^* + \mathcal{N}_\Omega(x^*)\text{.}
\end{equation}
Otherwise, there is a hyperplane separating the point $-F(x^*)$ and
the set $\{A^T\bar{\lambda}\,|\, \bar{\lambda} \in \mathbb{R}^l\} +
\mathcal{N}_\Omega(x^*)$. Namely, there is a vector $\omega \in
\mathbb{R}^{n}$ such that
\begin{subequations}
\begin{align}
\label{eq:seperating1}
\omega^TF(x^*) &<0\text{,}\\
\label{eq:seperating2}
A\omega &= 0 \text{,}\\
\label{eq:seperating3}
d^T\omega &\leq 0,\,\forall\, d\in \mathcal{N}_\Omega(x^*)\text{.}
\end{align}
\end{subequations}
From \eqref{eq:seperating2} and \eqref{eq:seperating3},
$\omega \in \mathcal{T}_\Omega(x^*) \cap
\mathcal{T}_\mathcal{X}(x^*)$.  Moreover, with Assumption
\ref{assum:basic}, $\omega \in \mathcal{T}_{\Omega\cap
\mathcal{X}}(x^*) = \mathcal{T}_\mathcal{K}(x^*)$ according to Lemma
\ref{lemma:tangent_cone}. Recalling the definition of the tangent
cone, there exist $x_k \in \mathcal{K}$ and $t_k>0$ such that
\begin{equation*}
x_k\to x^*,\,t_k\to 0, \text{ and } \lim_{k\to \infty}\frac{x_k-x^*}{t_k} = \omega\text{.}
\end{equation*}
Consequently,
\begin{equation*}
\omega^TF(x^*) = \lim_{k\to \infty}\frac{(x_k-x^*)^TF(x^*)}{t_k}\geq 0\text{,}
\end{equation*}
which contradicts to \eqref{eq:seperating1}. It completes the proof.
\end{proof}

Clearly, $X^* = \{x^*\}$ from Theorems \ref{thm:G2VI} and \ref{thm:KKT}, where $x^*$ is the unique variational GNE under Assumption \ref{assum:basic}.

Finally, it is time to give our convergence result.

\begin{theorem}\label{thm:smooth_convergence}
Under Assumptions \ref{assum:basic} and \ref{assum:graph}, system
\eqref{eq:algorithm2} is stable and converges to the set $X^*\times
\Lambda^*$, that is,
\begin{equation}\label{eq:convergence}
\left\{\begin{aligned}
&\lim_{t\to +\infty}\|\lambda_i(t)-\bar{\lambda}(t)\| = 0, \,\forall\, i=1,...,N\\
&\lim_{t\to +\infty} ||(x(t),\bar{\lambda}(t))-P_{X^*\times \Lambda^*}(x(t),\bar{\lambda}(t))|| = 0
\end{aligned}\right.
\end{equation}
with $\bar{\lambda}(t)$ defined in \eqref{eq:barlambda}.
\end{theorem}

\begin{proof}
Since $\dot{x} \in \mathcal{T}_{\Omega}(x)$, $x(t)\in
\Omega$ for all $t\geq 0$. Also, we have
\begin{multline}\label{eq:dotx}
\dot{x}_i(t) = P_{\Omega_i}(x_i - G_i(x_i,\sigma(x)) - \frac{\gamma}{N}A_i^T\bar{\lambda}(t))\\
 - x_i(t) + e_i(t)\text{,}
\end{multline}
where $\|e_i(t)\| = \|P_{\Omega_i}(x_i - G_i(x_i,\sigma(x)) -
\frac{\gamma}{N}A_i^T\bar{\lambda}(t))- P_{\Omega_i}(x_i -
G_i(x_i,\eta_i) - \frac{\gamma}{N}A_i^T\lambda_i(t))\|\leq
\|G_i(x_i,\sigma(x)-G_i(x_i,$ $\eta_i)\|+
\|\frac{\gamma}{N}A_i^T(\bar{\lambda}(t)-\lambda_i(t))\| \leq \bar k(\|\sigma(x)-\eta_i(t)\| +
\|\bar{\lambda}(t)-\lambda_i(t)\|)$, for some $\bar k>0$. Let $e(t)
= col\{e_1(t), ... , e_N(t)\}$. Hence, $e(t)$
vanishes exponentially according to Lemma \ref{thm:Equivalent_algorithm}.

Define
\begin{equation}\label{eq:variables}
\begin{aligned}
&\theta(t) \triangleq \begin{bmatrix}x(t)\\ \bar{\lambda}(t)\end{bmatrix}, &&
\widehat{F}(\theta) \triangleq \begin{bmatrix}F(x)+\frac{\gamma}{N}A^T\bar{\lambda}\\ -\frac{\gamma}{N}(Ax-b)\end{bmatrix},\\
&\Theta \triangleq \Omega \times \mathbb{R}^l, && \widehat{H}(\theta) \triangleq P_{\Theta}(\theta-\widehat{F}(\theta))\text{,}
\end{aligned}
\end{equation}
and let $\theta^* \in X^* \times \Lambda^*$. Consider the following Lyapunov function
\begin{equation}\label{eq:Lyapunov}
V(t) \triangleq  (\theta-\widehat{H}(\theta))^{T}\widehat{F}(\theta)-\frac{1}{2}\|\theta-\widehat{H}(\theta)\|^2 + \frac{1}{2}\left\|\theta
- \theta^*\right\|^2\text{.}
\end{equation}
It follows from Lemma \ref{lemma:projection_differential} that $V(t) \geq 0$ and
\begin{equation}\label{eq:dotV}
\dot{V}(t) = (\nabla_{\theta}V)^T \dot{\theta}(t) =  (\nabla_{\theta}V)^T(\widehat{H}(\theta)-\theta) + \hat{e}(t)\text{,}
\end{equation}
where,
\begin{equation*}
\begin{aligned}
\hat{e}(t) & = [e(t), \mathbf{0}_l]^T  \nabla_{\theta}V = e(t)^T \big[F(x)+\frac{\gamma}{N}A^T\bar{\lambda}(t) \\
& \quad + (\nabla F(x) -I_n)\big(x-P_\Omega(x-F(x)-\frac{\gamma}{N}A^T\bar{\lambda}(t))\\
& \quad  + \frac{\gamma}{N} A^T(Ax-b)\big) + x(t)-x^*\big].
\end{aligned}
\end{equation*}
Since $x(t)\in \Omega$ is bounded, $\bar{\lambda}(t)$ in \eqref{eq:barlambda} satisfies $\|\bar{\lambda}(t)\| \leq
K_1+K_2t$ for $t\geq 0$ with some constants $K_1,K_2>0$. Also, since $e(t)$ is
exponentially convergent, $\|\hat{e}(t)\| \leq
(K_3+tK_4)e^{-K_5t}$ for some constants $K_3,K_4,K_5>0$, which further implies
\begin{equation}\label{eq:eint}
\int_0^{+\infty} \|\hat{e}(t)\|dt <+\infty.
\end{equation}
On the other hand, after calculations, we have
\begin{equation*}
(\nabla_{\theta}V)^T(\widehat{H}(\theta)-\theta) = - W_1(\theta) -
W_2(\theta) - W_3(\theta) - W_4(\theta),
\end{equation*}
where
\begin{subequations}\label{eq:W}
\begin{align}
\label{eq:condition1}
W_1(\theta) & = (\theta^* - \widehat{H}(\theta))^T(\widehat{F}(\theta) + \widehat{H}(\theta) - \theta)\text{,}\\
W_2(\theta) & = (\theta - \theta^*)^T\widehat{F}(\theta^*)\text{,}\\
W_3(\theta) & = (\theta - \theta^*)^T(\widehat{F}(\theta)-\widehat{F}(\theta^*))\text{,}\\
W_4(\theta) & = (\widehat{H}(\theta) - \theta)^T \mathcal{J} \widehat{F}(\theta) (\widehat{H}(\theta) - \theta)\text{.}
\end{align}
\end{subequations}
It follows from \eqref{eq:pro_projection} that $W_1(\theta) \geq 0$.
Moreover,
\begin{equation}\label{eq:projection_stationary}
W_2(\theta) \in (x-x^*)^T(- \mathcal{N}_{\Omega}(x^*)) \subseteq
\mathbb{R}_+,\,\forall\,\theta\in \Theta.
\end{equation}
Furthermore, for any $\theta,\theta' \in \Theta$ with $x\neq x'$, we obtain
\begin{equation}\label{eq:monotone_equation}
(\theta - \theta')^T (\widehat{F}(\theta)-\widehat{F}(\theta')) = (x-x')^T(F(x) - F(x')) >0\text{,}
\end{equation}
since $F(x)$ is strictly monotone. Then $W_3(\theta) \geq 0$.
Also, because $\mathcal{J} \widehat{F}(\theta)$ is positive semidefinite,
$W_4(\theta) \geq 0$. Therefore,
\begin{equation}\label{eq:V_e}
\dot{V}(t) \leq -W_3(t) +
\hat{e}(t)\text{.}
\end{equation}
Then it follows from \eqref{eq:eint} that
\begin{equation*}
\begin{aligned}
0 \leq \int_{0}^{+\infty} W_3(t)dt & \leq  V(0)-\limsup_{t\to+\infty} V(t)\\
&\quad  + \int_{0}^{+\infty} \hat{e}(t)dt <+\infty\text{,}
\end{aligned}
\end{equation*}
Consequently, $W_3(t)\to 0$ as $t\to +\infty$, which implies $\lim_{t\to +\infty} x(t) = x^*$. Moreover, $\limsup_{t\to+\infty} V(t)<+\infty$, which implies that $\bar{\lambda}(t)$ is bounded. Therefore, the positive limit set
$X^+\times \Lambda^+ \neq \emptyset$. It follows from
\eqref{eq:barlambda} that $\lim_{t\to+\infty}\dot{\bar{\lambda}}(t)
= Ax^* - b = 0$. Furthermore, $\dot{x}(t)$ is
uniformly continuous because the trajectory of
\eqref{eq:algorithm2} is absolutely continuous and the righthand
side of the differential equation with respect to $x(t)$ in
\eqref{eq:algorithm2} is uniformly continuous in $t$. Since $x(t)$ is convergent, $\lim_{t\to+\infty}\dot{x}(t) = 0$ by the well-known Barbalat's lemma. Thus,
\eqref{eq:convergence} holds according to Lemmas
\ref{thm:Equivalent_algorithm} and \ref{eq:positive_limit}.
\end{proof}

\begin{rem}
Our algorithm can be viewed as a distributed perturbed projected dynamics with the exponentially vanishing perturbation term $e_i(t)$ in \eqref{eq:dotx}. Although some projected dynamics without any perturbation was studied, e.g., in \cite{Yi2016Initialization}, the analysis for the perturbed one is novel.
\end{rem}

\section{Numerical Examples}
Two numerical examples are given in this section.

\subsection{Nash-Cournot Game}
Consider a Nash-Cournot game played by $N$ competitive firms to
produce a kind of commodity.  For $i\in\mathcal{V}=\{1,...,N\}$, firm $i$ chooses $x_i\in \Omega_i$ as the quantity of the
commodity to produce and has the cost function as $\vartheta_i(x_i,\sigma) =
(c_i- p(\sigma))x_i$, where $c_i$ is the production price of firm
$i$, and $p = d - N\sigma(x)$ is the market price determined by the
aggregation function $\sigma(x) = \frac{1}{N}\sum_{j\in \mathcal{V}}
x_j^2$. Our numerical setting is as follows.
\begin{enumerate}[(i)]
\itemsep = 0pt \parskip = 0pt
\item $N = 20$ and $\mathcal{V} = \{1,...,20\}$.
\item For each firm $i\in \mathcal{V}$, $\Omega_i = [0, 20]$, $c_i= 10+20(i-1)$ and $d=1200$.
\item Firms from $i = 1$ to $i = 10$ share a scare resource as $\sum_{i=1}^{10}x_i = 20$.
\item the communication graph $\mathcal{G}(t)$ is time-varying and randomly generated.
\item Parameter setting of our algorithm is $\alpha = 20, \beta =400$ and $\gamma = 20$.
\end{enumerate}
Figure \ref{gcds1} shows the convergence to the NE (the upper one) and GNE (the lower one), which illustrates the effectiveness of our algorithm.

\begin{figure}
  \centering
  \includegraphics[width=8.1cm]{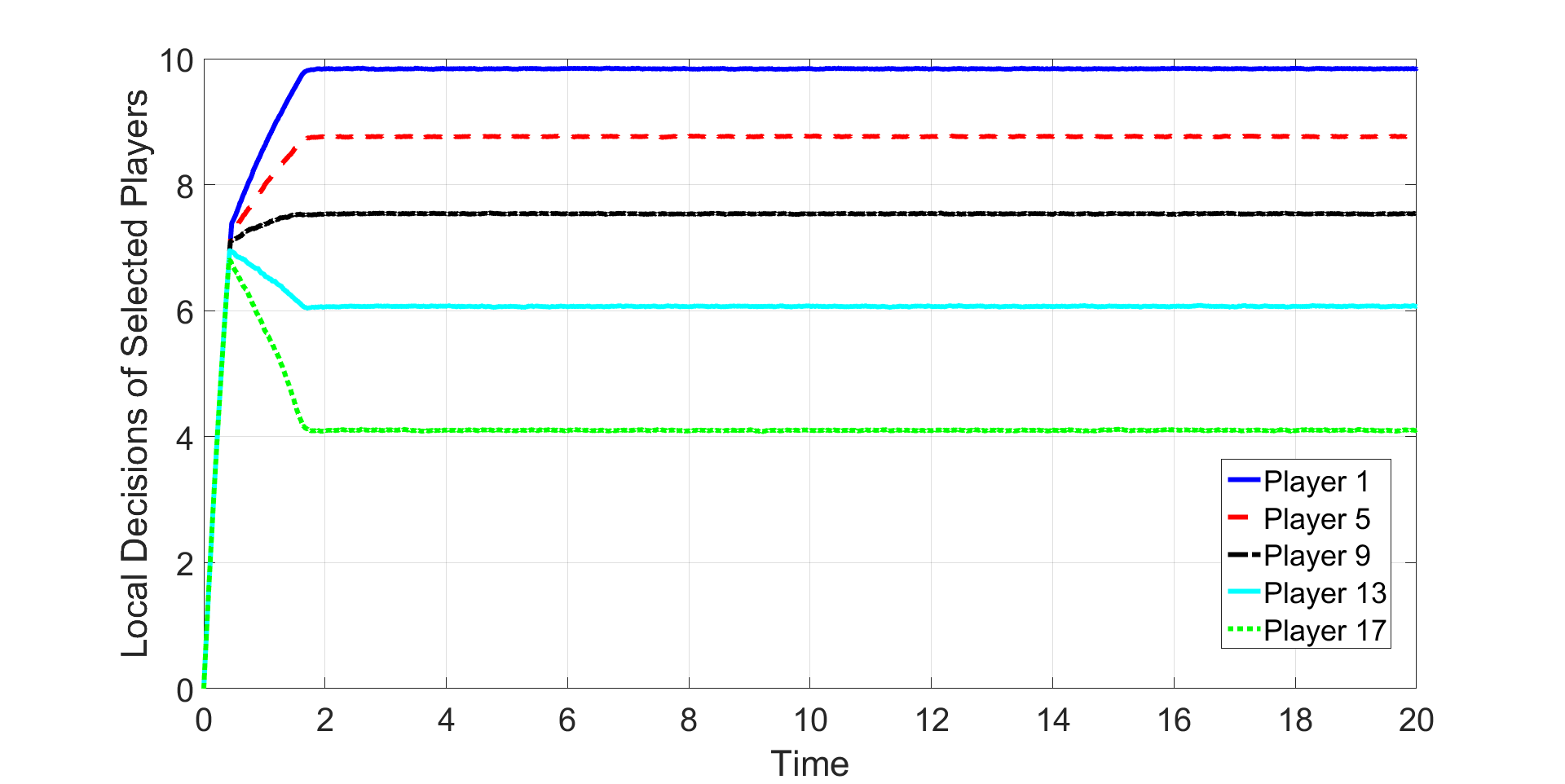}\\
  \includegraphics[width=8.1cm]{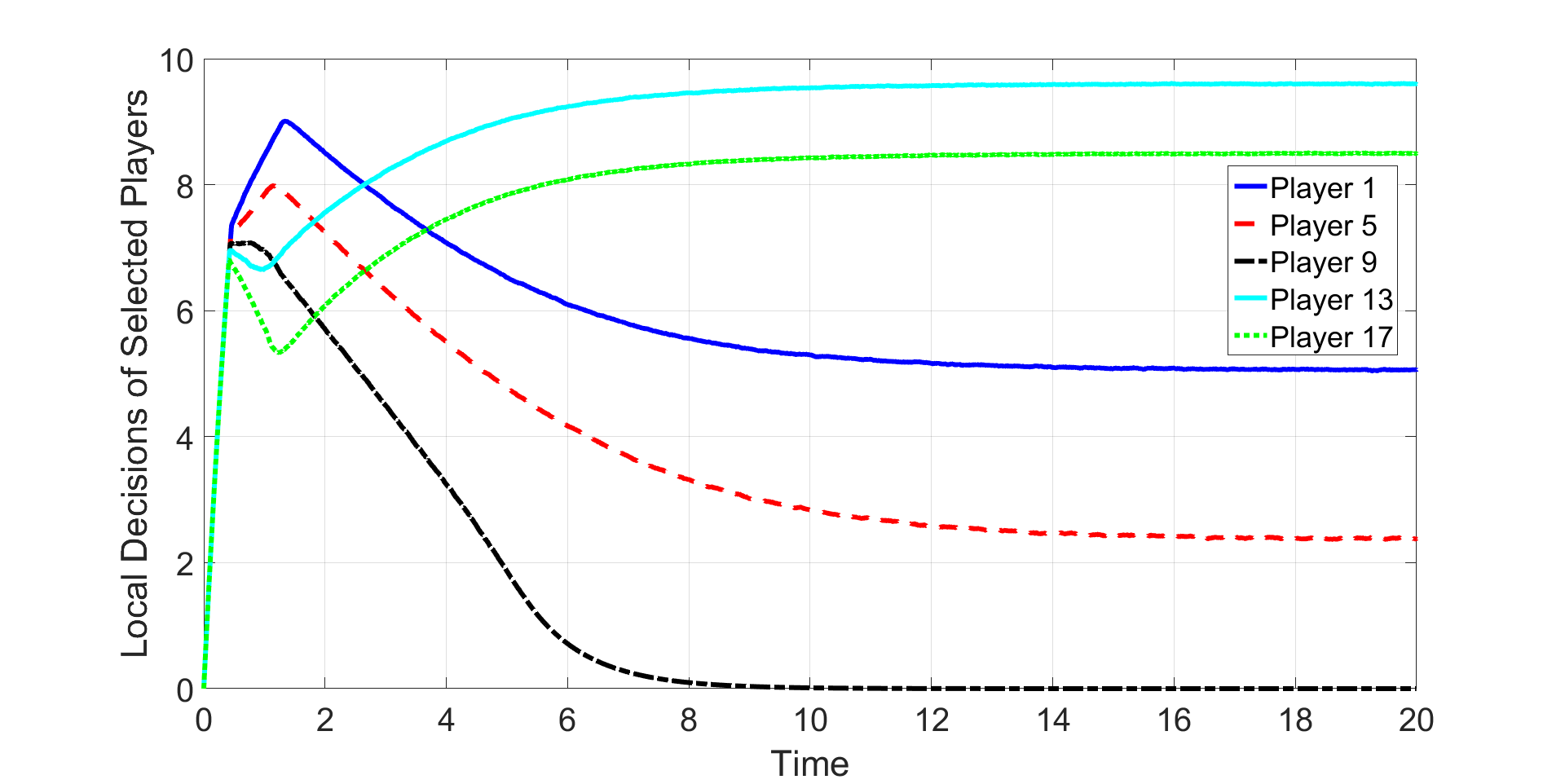}\\
  \caption{The trajectories of the strategy profile of Nash-Cournot game without (the upper one) and with (the lower one) the linear coupled constraint.}\label{gcds1}
\end{figure}

\subsection{Demand Response Management}
Consider $N$ electricity users with the demand of energy consumption.
For each $i\in \mathcal{V}$, $x_i \in [\underline{r}_i,\bar{r}_i]$
is the energy consumption of the $i$th user and $C_i(x_i,\sigma(x))$
is the cost function in the following form
\begin{equation*}
C_i(x_i,\sigma(x)) = k_i (x_i-\chi_i)^2 + P(\sigma(x))x_i\text{,}
\end{equation*}
where $k_i$ is constant and $\chi_i$ is the nominal value of energy
consumption for $i=1,...,N$, with $P(\sigma(x)) = aN\sigma(x)+p_0$
and $\sigma(x) = \frac{1}{N}\sum_{i\in\mathcal{V}}x_i$. We adopt the
same numerical setting as given in \cite{Ye2017Game} as follows:
\begin{enumerate}[(i)]
\itemsep = 0pt \parskip = 0pt
\item $N = 5$, $k_i = 1$, $a = 0.04$, $p_0 = 5$.
\item $\chi_1 = 50, [\underline{r}_1,\bar{r}_1] = [45,55]; \quad \chi_2 = 55, [\underline{r}_2,\bar{r}_2] = [44,66]; \quad \chi_3 = 60, [\underline{r}_3,\bar{r}_3] = [46,72]; \quad \chi_4 = 65, [\underline{r}_4,\bar{r}_4] = [52,78]; \quad \chi_5 = 70, [\underline{r}_5,\bar{r}_5] = [56,84]$.
\end{enumerate}
The NE of the game has been calculated in \cite{Ye2017Game} as
$x^* = [45,46.4,51.3,56.2,61.1]^T$. Note that the total energy
consumption at this NE may be too far from the normal value because
$\sum_{i=1}^5 x_i^* = \sum_{i=1}^5 \chi_i -40$. Here, we further
impose the following linear constraint $\sum_{i=1}^5 x_i =
\sum_{i=1}^5\chi_i - 25$. The communication graph is randomly
generated and the parameters in
our algorithm are given as $\alpha = 30, \beta =100$, and $\gamma = 2$. Then the GNE is obtained as $x^* =
[45.2,50.1,55,59.9,64.8]^T$. Figure \ref{gcds3} shows the convergence to the NE (the upper one) and GNE (the lower one), which again illustrates our algorithm.

\begin{figure}
  \centering
  \includegraphics[width=8.1cm]{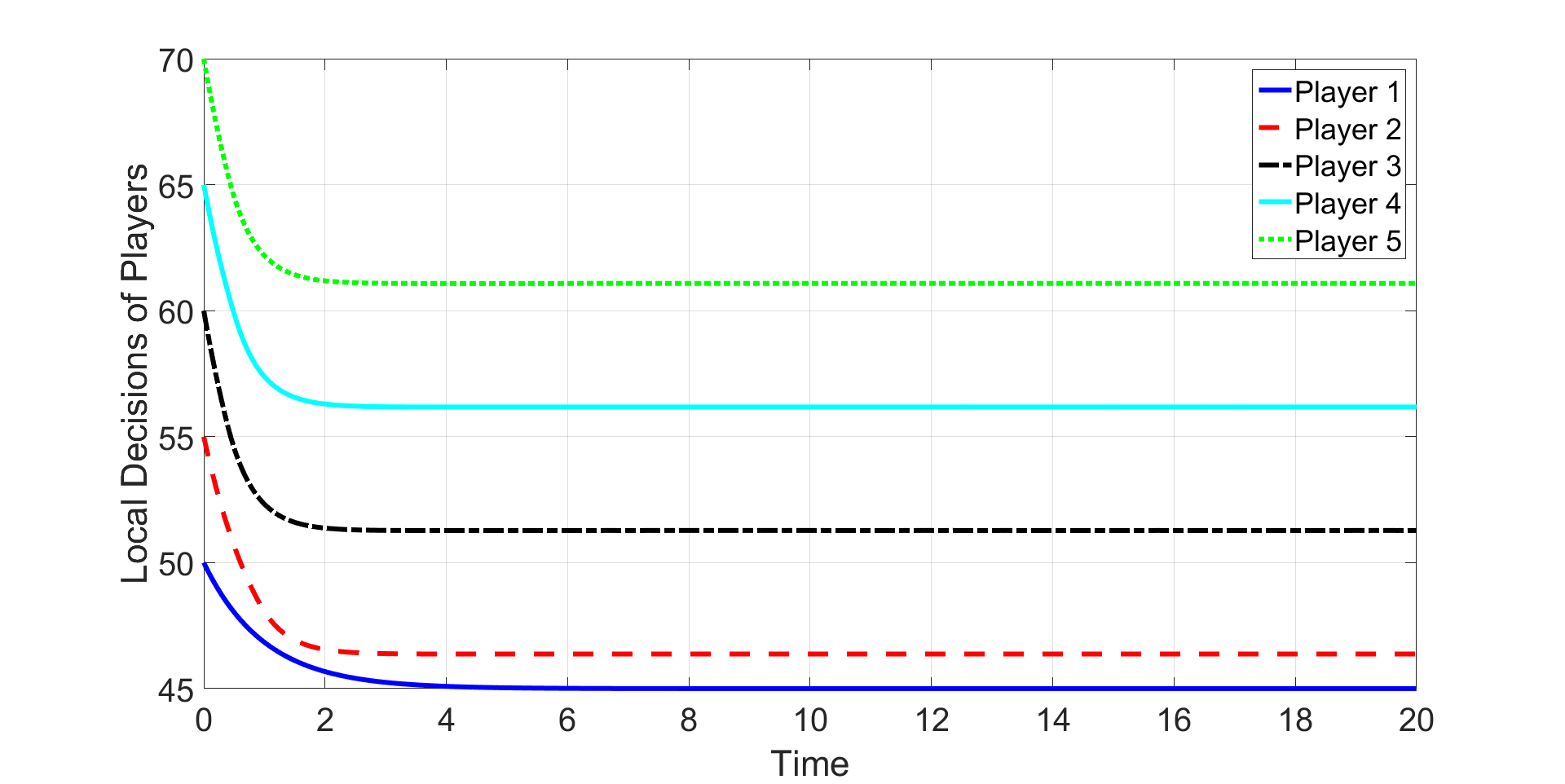}\\
  \includegraphics[width=8.1cm]{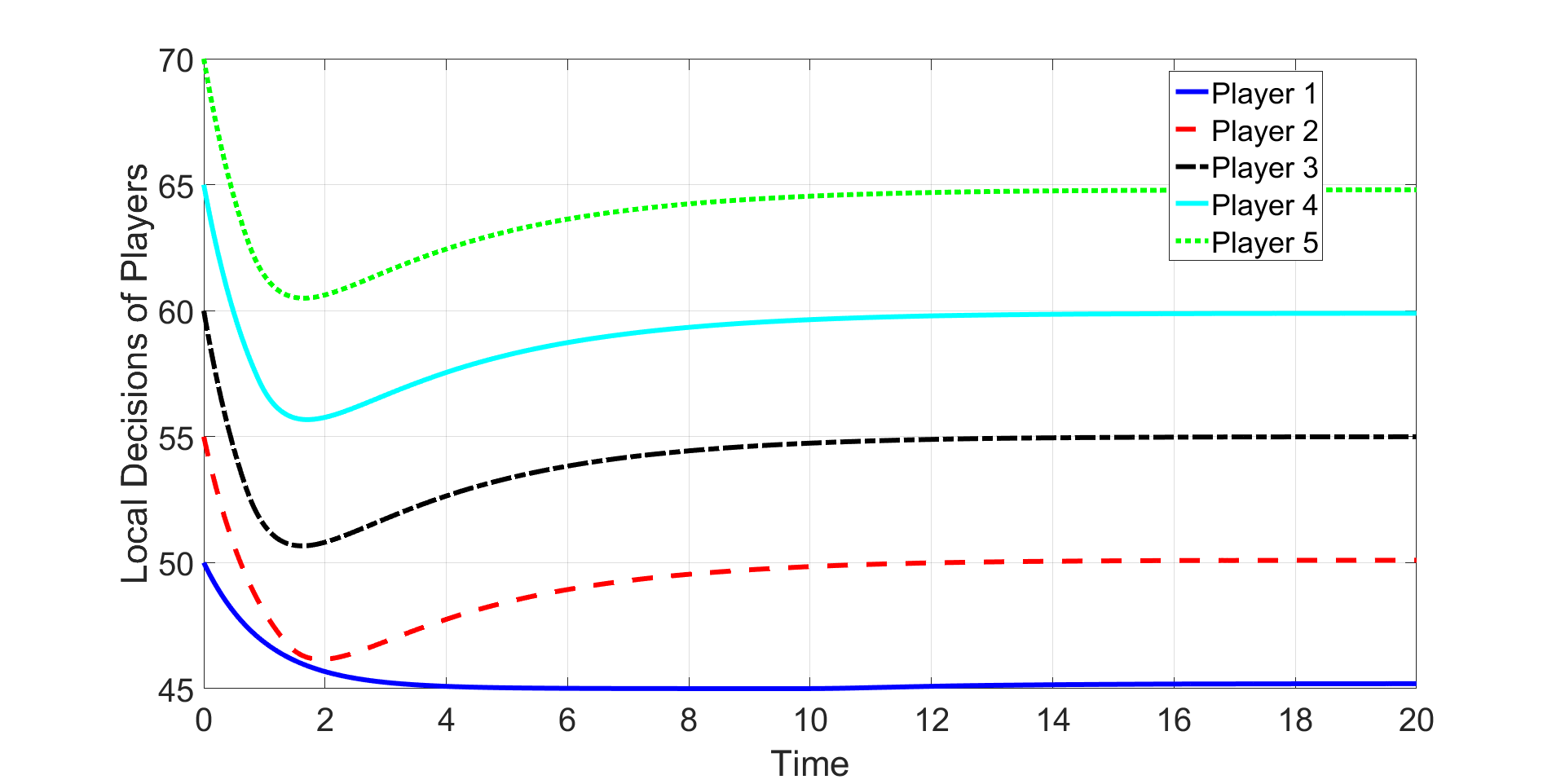}\\
  \caption{The trajectories of the strategy profile of demand response management game without (the upper one) and with (the lower one) the linear coupled constraint.}\label{gcds3}
\end{figure}

\section{Conclusions}
In this paper, aggregative games with linear coupled constraints
were considered and a distributed continuous-time projection-based
algorithm was proposed for the GNE seeking. The correctness and
convergence of the proposed non-smooth algorithm were proved by virtue of variational
inequalities and Lyapunov functions, and moreover, two numerical
examples were given for illustration.

\bibliographystyle{dcu}        
\bibliography{E:/HongLab/bib/refference0,E:/HongLab/bib/refference1}
\end{document}